\documentclass[letterpaper, 10 pt, conference]{ieeeconf}
\IEEEoverridecommandlockouts
\usepackage{amsmath,amssymb,amsfonts,bm,mathtools,yhmath}

\usepackage{amsthm}

\usepackage{cite}

\usepackage{booktabs}
\usepackage{multirow}
\usepackage{array}
\usepackage{graphicx}
\usepackage{xcolor}
\usepackage{tikz}
\usepackage{pgfplots}
\pgfplotsset{compat=1.18}
\usetikzlibrary{arrows.meta,positioning,fit,calc,shapes.geometric,backgrounds,decorations.pathreplacing}
\usepackage[hidelinks,breaklinks=true]{hyperref}
\usepackage{url}
\theoremstyle{thmstyleone}%
\newtheorem{theorem}{Theorem}[section]

\newtheorem{proposition}[theorem]{Proposition}

\theoremstyle{definition}

\newtheorem{assumption}{Assumption}

\theoremstyle{remark}

\newcommand{\Voc}{\mathcal{V}}
\newcommand{\Vstar}{\Voc^{*}}
\newcommand{\R}{\mathbb{R}}

\newcommand{\LN}{\operatorname{LN}}
\newcommand{\MHA}{\operatorname{MHA}}
\newcommand{\MLP}{\operatorname{MLP}}
\newcommand{\GELU}{\operatorname{GELU}}
\newcommand{\softmax}{\operatorname{softmax}}
\newcommand{\Concat}{\operatorname{Concat}}

\newcommand{\one}{\mathbf{1}}

\DeclareMathOperator*{\argmax}{arg\,max}

\begin{document}

\title{\LARGE \bf
On the Multi-Index, Multi-Rate, and Multi-Phase \\ Dynamics
of Decoder-Only Language Models: \\ A Unified Hybrid Framework
for Generative and Agentic Systems
}

\author{Ali Pakniyat, \textit{Member, IEEE}%
\thanks{A.~Pakniyat is with the Multi-Modal Multi-Agent Control ($\mathcal{M}^3\!\mathcal{A\,C}$) Lab (e-mail: {\tt m3ac.lab@gmail.com}).}%
}

\maketitle
\thispagestyle{empty}
\pagestyle{empty}

%=======================================================================
\begin{abstract}

Large language models (LLMs) are increasingly deployed as computational engines in autonomous decision-making and planning loops, yet their systems and control treatment remains hindered by architectural simplifications, index conflations, and informal descriptions of tool interactions. This paper presents a control-theoretic formulation of decoder-only language models as multi-index, multi-rate systems, and sets the stage for a stochastic hybrid systems framework to govern the multi-phase dynamics of agentic tool interaction. We formalize the architecture across three hierarchically coupled evolution indices: 
(i) an ultrafast feedforward cascade of transformer blocks across layer depth, where layer normalization is cast as a spherical projection and key--value caching is proven to be an exact internal state realization via causal prefix invariance;
(ii) an uncontrolled stochastic difference recursion over token generation steps, where finite context truncation induces a time-homogeneous Markov chain; and
(iii) an autonomous mode-switching mechanism governing transitions between token generation and tool execution regimes, where tool invocations are triggered upon trajectory arrival at switching manifolds, followed by exogenous state jump maps that augment the context string with external observations. By defining a prompt-dependent evaluator over successive evaluable claims, we obtain a task-level error process whose fault-free histories and expected error growth admit bounds under conditional fault-hazard and error-drift assumptions.

\end{abstract}

%=======================================================================
%=======================================================================
\section{Introduction}
\label{sec:intro}
%=======================================================================
The transformer architecture, introduced originally as an encoder-decoder system for sequence transduction \cite{vaswani2017attention}, has converged into the ubiquitous ``decoder-only'' configuration in modern foundation models, including OpenAI's GPT series \cite{radford2019gpt2}, Meta's LLaMA, and Mistral. Unlike encoder-decoder architectures driven by exogenous reference inputs (such as a source text in machine translation), a decoder-only model operates as an unforced, causal sequence generator that evolves purely via recursive token concatenation. Within each token-generation step, the circulating sequence of discrete tokens is mapped via fixed embedding operators into an initial continuous tensor state, which propagates through a cascade of spatial- and channel-mixing transformer blocks under a lower-triangular causal attention kernel. In modern autonomous and agentic applications, this sequence generator operates across two primary modalities: \emph{generative} open-loop rollout, wherein text evolves through unforced recursive token emission, and \emph{agentic} tool-augmented interaction \cite{yao2023react,schick2023toolformer}, wherein internal generation halts to coordinate with external computational engines, databases, or physical robotic actuators \cite{ahn2022saycan}.

Recognizing that generative foundation models are increasingly embedded inside critical decision-making loops, the systems and control community has initiated several lines of research to analyze transformers using rigorous dynamical tools. Prompt engineering has been cast as an optimal control problem with an explicit notion of reachability \cite{bhargava2023magicword}, depth layers have been modeled as continuous-time interacting particle systems and consensus dynamics on manifolds \cite{geshkovski2023clusters,silvestre2024attention,abella2025consensus}, attention routing has been studied through the lens of Markov chains \cite{ildiz2024markov}, next-token prediction has been interpreted via nonlinear filtering and stochastic optimal control duality \cite{chang2025dualfilter,chang2025can, kudre2026transformer}, and robotic deployment safety has been investigated using control barrier functions \cite{ravichandran2025guardrails}.

Despite this fruitful expansion, existing representations in the literature exhibit notable modeling simplifications and structural ambiguities. In reachability-based prompt formulations such as \cite{bhargava2023magicword}, the entire multi-layer transformer network is collapsed into an abstract static operator; this abstraction produces modeling discrepancies where the control input is introduced as a continuation token at one stage, yet treated as an amended initial condition in the state space at another. Conversely, continuous-time depth formulations \cite{geshkovski2023clusters,silvestre2024attention,abella2025consensus} provide valuable convergence results by viewing layer depth as ordinary differential equations, yet they deliberately isolate the self-attention operator under idealized conditions---frequently omitting layer normalizations, projection matrices, and the position-wise Multi-Layer Perceptrons (MLPs) that constitute over two-thirds of the network's non-embedding parameters. More fundamentally, control expositions routinely conflate the deterministic, finite-horizon feedforward evaluation across layer depth $\ell \in \{1,\dots,L\}$ with the uncontrolled stochastic difference recursion that generates tokens across generation time $t \in \mathbb{N}$; these two processes evolve on distinct index sets, inhabit different mathematical spaces, and possess fundamentally different dynamical properties. Finally, in the machine learning literature on agentic tool use \cite{yao2023react,schick2023toolformer}, tool invocation is treated as an informal software wrapper or string-delimited prompt hook. Rather than recognizing these interactions as formal dynamical phenomena, existing control treatments lack a unified framework to model tool-augmented generation as a hybrid system with discrete modes, state jump maps, and contractive error dynamics.

%=======================================================================
% FULL TIKZ FIGURE 1: TWO PANELS
%=======================================================================
\begin{figure*}[t]
\centering

%-----------------------------------------------------------------------
% Panel (a): full decoder-only architecture
%-----------------------------------------------------------------------
\resizebox{\textwidth}{!}{
\begin{tikzpicture}[font=\small,
  box/.style={draw,rounded corners=2pt,minimum height=9mm,minimum width=15mm,align=center,fill=blue!4},
  blk/.style={draw,rounded corners=2pt,minimum height=13mm,minimum width=20mm,align=center,fill=orange!8},
  ar/.style={-{Latex[length=2mm]},thick},
  lbl/.style={font=\scriptsize,text=black!65}]

\node[font=\bfseries, anchor=west] at (10,-1.7)
{(a) Decoder-only architecture};

\node[box] (tok) {Tokenizer\\$s_t=(w_1,\ldots,w_n)$};
\node[box,right=9mm of tok] (emb) {Embedding \eqref{eq:embed_matrix}\\$H^{(0)}=U(s_t) W_e + P_n W_p$};
\node[blk,right=11mm of emb] (b1) {Transformer \\ Block $1$\\ $\mathcal{B}_1$};
\node[right=6mm of b1] (dots1) {$\cdots$};
\node[blk,right=8mm of dots1] (b2) {Transformer \\ Block $\ell$\\ $\mathcal{B}_{\ell}$};
\node[right=6mm of b2] (dots) {$\cdots$};
\node[blk,right=9mm of dots] (bL) {Transformer \\ Block $L$\\ $\mathcal{B}_L$};
\node[box,right=9mm of bL] (lnf) {$\LN_f$ and\\unembedding \eqref{eq:logits_readout}};
\node[box,right=14mm of lnf] (sm) {$\softmax$ Kernel \eqref{eq:temp_distribution} \\ $p_\theta(w_{n+1} = v_k \mid s_t)$};
\node[box,right=9mm of sm,fill=green!6] (samp) {Selection\\ $w_{t+1}$};
\draw[ar] (tok)--(emb);
\draw[ar] (emb)--node[lbl,above]{$H^{(0)}$}(b1);
\draw[ar] (b1)--node[lbl,above]{$\,H^{(1)}$}(dots1);
\draw[ar] (dots1)--node[lbl,above]{$H^{(\ell-1)}\,$}(b2);
\draw[ar] (b2)--node[lbl,above]{$\,H^{(\ell)}$}(dots);
\draw[ar] (dots)--node[lbl,above]{$H^{(L-1)}\,$}(bL);
\draw[ar] (bL)--node[lbl,above]{$H^{(L)}$}(lnf);
\draw[ar] (lnf)--node[lbl,above]{$z(s_t)$}(sm);
\draw[ar] (sm)--(samp);
\coordinate (down) at ($(samp.south)+(0,-9mm)$);
\draw[ar,dashed] (samp.south)--(down) -- node[lbl,above,pos=0.62]{Next token generation: $s_{t+1} = \Pi_T(s_t \circ w_{t+1})$} ($(tok.south)+(0,-9mm)$) -- (tok.south);
\draw[decorate,decoration={brace,amplitude=4pt},black!55] ($(b1.north west)+(-1mm,3mm)$) -- ($(bL.north east)+(1mm,3mm)$) node[midway,above=2pt,lbl]{$L$ identical pre-normalized transformer blocks (e.g., $L=12$ for GPT-2-small)};
\end{tikzpicture}}

\vspace{2mm}

%-----------------------------------------------------------------------
% Panel (b): transformer block
%-----------------------------------------------------------------------
\resizebox{\textwidth}{!}{%
\begin{tikzpicture}[
  font=\scriptsize,
  >=Latex,
  node distance=5mm and 10mm,
  box/.style={
    draw,
    rounded corners=2pt,
    align=center,
    minimum height=7mm,
    minimum width=24mm,
    fill=white,
    inner sep=2pt
  },
  qkv/.style={
    draw,
    rounded corners=2pt,
    align=center,
    minimum height=7mm,
    minimum width=23mm,
    fill=blue!8,
    inner sep=2pt
  },
  qkvshadow/.style={
    draw=blue!35,
    rounded corners=2pt,
    minimum height=8.5mm,
    minimum width=27mm,
    fill=blue!3,
    inner sep=2pt
  },
  attn/.style={
    draw,
    rounded corners=2pt,
    align=center,
    minimum height=7mm,
    minimum width=27mm,
    fill=green!8,
    inner sep=2pt
  },
  attnshadow/.style={
    draw=green!50!black!35,
    rounded corners=2pt,
    minimum height=7mm,
    minimum width=27mm,
    fill=green!3,
    inner sep=2pt
  },
  mixshadow/.style={
    draw=orange!60!black!35,
    rounded corners=2pt,
    minimum height=7.4mm,
    minimum width=30mm,
    fill=orange!4,
    inner sep=2pt
  },
  mlp/.style={
    draw,
    rounded corners=2pt,
    align=center,
    minimum height=7mm,
    minimum width=27mm,
    fill=purple!8,
    inner sep=2pt
  },
  sum/.style={
    circle,
    draw,
    fill=yellow!20,
    minimum size=5.5mm,
    inner sep=0pt
  },
  arrow/.style={->, semithick},
  flow/.style={
    midway,
    above,
    font=\scriptsize,
    align=center,
    text=black!75,
    inner sep=1pt
  }
]
\node[font=\small\bfseries, anchor=west] at (4.2,-2.7)
{(b) Expanded pre-normalized transformer block $\mathcal{B}_\ell$};

\coordinate (block_in) at (-4.0,0.8);
\coordinate (block_out) at (20.2,-0.93);

% Sublayer 1: Attention
\node[box] (ln1) at (0,0.8)
{Layer Normalization \eqref{eq:ln_def} \\[3pt] $z_i^{(\ell)}\!=\!\LN_{\ell,1}\!(h_i^{(\ell-1)})$};

\node[qkvshadow] at ($(3.5,2.02)+(2.5pt,2.5pt)$) {};
\node[qkvshadow] at ($(3.5,2.02)+(1.25pt,1.25pt)$) {};
\node[qkv] (Q) at (3.5,2.0)
{Query Projection \eqref{eq:qkv_proj} \\[3pt] $Q^{(\ell, m)} \!=\! Z^{(\ell)} W_Q^{(\ell, m)}$};

\node[qkvshadow] at ($(3.5,0.8)+(2.5pt,2.5pt)$) {};
\node[qkvshadow] at ($(3.5,0.8)+(1.25pt,1.25pt)$) {};
\node[qkv] (K) at (3.5,0.8)
{Key Projection \eqref{eq:qkv_proj} \\[3.5pt] $K^{(\ell, m)} \!=\! Z^{(\ell)} W_K^{(\ell,m)}$};

\node[qkvshadow] at ($(3.5,-1.4)+(2.5pt,2.5pt)$) {};
\node[qkvshadow] at ($(3.5,-1.4)+(1.25pt,1.25pt)$) {};
\node[qkv] (V) at (3.5,-1.4)
{Value Projection \eqref{eq:qkv_proj} \\[3.5pt] $V^{(\ell, m)} \!=\! Z^{(\ell)} W_V^{(\ell,m)}$};

\node[attnshadow] at ($(7.0,0.82)+(2.5pt,2.5pt)$) {};
\node[attnshadow] at ($(7.0,0.82)+(1.25pt,1.25pt)$) {};
\node[attn] (score) at (7.0,0.8)
{Causal Affinity Kernel\\$S^{(\ell,m)}$: \eqref{eq:scores}};

\node[attnshadow] at ($(score.south)+(0,-9.5mm)+(2.5pt,2.5pt)$) {};
\node[attnshadow] at ($(score.south)+(0,-9.5mm)+(1.25pt,1.25pt)$) {};
\node[attn, below=6mm of score] (soft)
{Row-wise softmax\\$A^{(\ell,m)}$: \eqref{eq:attention_matrix}};

\node[mixshadow] at ($(10.9,-1.4)+(2.5pt,2.5pt)$) {};
\node[mixshadow] at ($(10.9,-1.4)+(1.25pt,1.25pt)$) {};
\node[box, fill=orange!12] (v_mix) at (10.9,-1.4)
{Value Mixing\\[3pt] $O^{(\ell,m)}=A^{(\ell,m)}V^{(\ell,m)}$};

\node[box] (mha_out) at (10.9,1.8)
{Head Concat. \& Projection \eqref{eq:mha_output}\\[3pt]
$\MHA_\ell = [O^{(\ell,m)}]_{m=1}^{N_h} W_O^{(\ell)} + \one_n b_O^{(\ell)}$};

\node[sum] (sum1) at (14.4,1.8) {$+$};

% Sublayer 2: MLP
\node[box] (ln2) at (14.4,0.4)
{Layer Normalization \eqref{eq:ln_def} \\[3pt] $R^{(\ell)}=\LN_{\ell,2}(\widetilde{H}^{(\ell)})$};

\node[mlp, below=6mm of ln2] (mlper)
{Multi-Layer Perceptron\\$\MLP_\ell$: \eqref{eq:mlp_matrix}};

\node[sum, right=17mm of mlper] (sum2) {$+$};

% Enclosing Block Box
\begin{scope}[on background layer]
  \draw[draw=orange!80!black, line width=0.8pt, rounded corners=4pt, fill=orange!8] (-1.8, -2.4) rectangle (18.2, 3.1);
  \node[anchor=north east, font=\scriptsize\bfseries, text=orange!90!black] at (2.6, -1.8) {\large Transformer Block $\mathcal{B}_\ell$};
\end{scope}

% Boundary & Internal Connections
\draw[arrow] (block_in) -- node[flow, pos=0.37] {$H^{(\ell-1)}\in\R^{n\times d}$} (ln1);
\draw[arrow] (sum2) -- node[flow, pos=0.6] {$H^{(\ell)}\in\R^{n\times d}$} (block_out);
\draw[arrow] (ln1.east) -- ++(0.55,0) |- (Q.west);
\draw[arrow] (ln1.east) -- node[flow, pos=0.34] {$Z^{(\ell)}$} (K.west);
\draw[arrow] (ln1.east) -- ++(0.55,0) |- (V.west);
\draw[arrow] (Q.east) -| (score.north);
\draw[arrow] (K.east) -- (score.west);
\draw[arrow] (score) -- (soft);
\draw[arrow] (soft.east) -| node[flow, above=2pt, pos=0.5] {$A^{(\ell,m)}$} ($(v_mix.north west)+(0.5,0)$);
\draw[arrow] (V.east) -- node[flow, pos=0.75, below=1pt] {$V^{(\ell,m)}$} (v_mix.west);
\draw[arrow] (v_mix) -- (mha_out);
\draw[arrow] (mha_out) -- node[flow,above=1pt] {$\MHA_\ell$} (sum1);
\draw[arrow] ($(ln1.west)+(-0.35,0)$) -- ++(0,1.9) -- node[flow, above] {Residual stream: $H^{(\ell-1)}$} ($(sum1.north)+(0,0.6)$) -- (sum1.north);
\draw[arrow] (sum1) -- node[flow, right, pos=0.5] {$\widetilde{H}^{(\ell)}$} (ln2);
\draw[arrow] (ln2) -- node[flow, right, pos=0.5] {$R^{(\ell)}$} (mlper);
\draw[arrow] (mlper) -- node[flow, above, pos=0.5] {$\MLP_\ell(R^{(\ell)})$} (sum2);
\draw[arrow] ($(sum1.south)+(0,-0.15)$) -- ++(3.34,0) -- node[flow, left, pos=1] {$\widetilde{H}^{(\ell)}$} ($(sum2.north)+(0,0.5)$) -- (sum2.north);

\end{tikzpicture}}
\vspace{-6mm}
\caption{System architecture of a decoder-only transformer. (a) Global generation recursion across time $t \in \mathbb{N}$ and deterministic depth cascade over $\ell \in \{1, \dots, L\}$. (b) Detailed state-space structure of transformer block $\mathcal{B}_\ell$ showing normalization projections, parallel multi-head causal attention routing, and position-decoupled channel mixing.}
\label{fig:architecture}
% \vspace{-6mm}
\end{figure*}
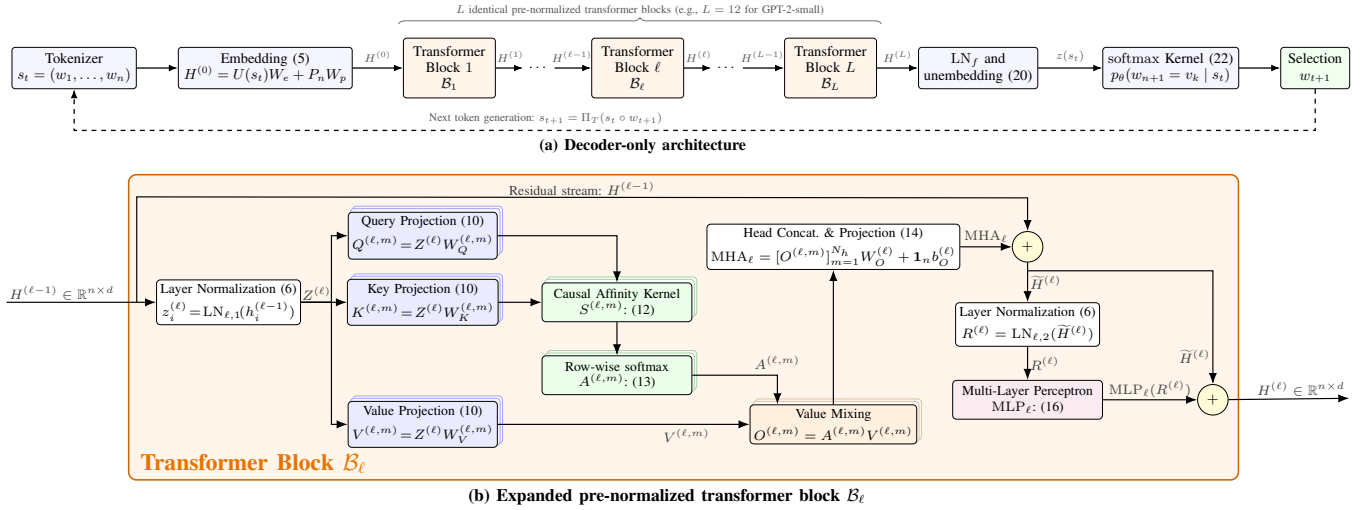

This paper eliminates these ambiguities by establishing an axiomatic, dimensionally exact control-theoretic framework that unifies both generative and agentic decoder-only language models. The architecture is formulated as a multi-index, multi-rate, and multi-phase dynamical system:

\textbf{1) Multi-Index Structure:} The system coordinates three hierarchically coupled index sets: (i) a spatial depth cascade $\ell \in \{1,\dots,L\}$, (ii) an uncontrolled discrete-time stochastic recursion over generation clock $t \in \mathbb{N}$, and (iii) a sequence of mode switching events indexed by $k \in \mathbb{N}$ marking the incurred number of transitions between discrete operational modes $q \in Q = \{q_{G}, q_{E}\}$.

\textbf{2) Multi-Rate Timing:} The indices operate across fundamentally distinct physical time scales: (i) an ultrafast micro-scale governing the feedforward layerwise tensor updates across $\ell \in \{1,\dots,L\}$ at hardware clock speed; (ii) an intermediate, nearly periodic (almost fixed duration) token-generation period marked by \mbox{$t \in \mathbb{N}$}, where each token step requires the complete spatial depth cascade to finish and memory bandwidth to clear; and (iii) an asynchronous macro-scale governing the switching event indices $k \in \mathbb{N}$, which occur at a~priori unknown token generation instants $t_k$ whenever the trajectory reaches a switching manifold representing external tool calls and execution completions.

\textbf{3) Multi-Phase Dynamics:} The dynamic evolution of the system is governed by two discrete modes $q \in Q = \{q_{G}, q_{E}\}$: an uncontrolled autoregressive generation mode $q_{G}$ (where the sequence recursion evolves via recursive token emission) and an external tool update mode $q_{E}$ (where internal generation is suspended while external computational or physical actions evolve). The interplay between these modes is governed by two instantaneous hybrid mechanisms: (i)~an~autonomous mode switching $q_{G} \to q_{E}$ triggered whenever the context trajectory hits the switching manifold $m_{q_{G}, q_{E}}$ (upon emitting a tool-call or action token), and (ii)~a~return transition $q_{E} \to q_{G}$ governed by a state jump map $\Phi_{q_{E}, q_{G}}$ that augments the context string state with the tool execution output.

By establishing this multi-index, multi-rate, and multi-phase foundation, the paper bridges the gap between static layerwise inference and dynamic temporal feedback, while utilizing stochastic hybrid systems frameworks, particularly  \cite{pakniyat2016stochastic,pakniyat2017relation}, as the mathematical foundation for agentic tool~use.

The remainder of this paper is organized as follows:
Section~\ref{sec:spaces} introduces the sequence and context state
spaces, the truncation operator, and the three decoupled dynamical
phases indexed by layer depth $\ell$, token-generation time $t$, and
switching time $k$. Section~\ref{sec:cascade} develops the
deterministic layerwise cascade, including the canonical-basis
representation of token and positional embeddings, the spherical
projection geometry of layer normalization, the spatial and channel
mixing operators, and the causal prefix invariance that makes
Key--Value (KV) caching an exact internal state realization.
Section~\ref{sec:generation} formulates the uncontrolled stochastic
generation recursion, establishes its time-homogeneous finite-state 
Markov representation under context truncation, and shows that greedy
trajectories are ultimately periodic. It then introduces the
prompt-dependent evaluation framework and derives bounds on fault-free
claim histories and expected error growth under conditional
fault-hazard and error-drift assumptions. Section~\ref{sec:conclusion} provides concluding remarks.

%=======================================================================
\section{State Spaces and Decoupled System Phases}
\label{sec:spaces}
%=======================================================================
A prerequisite for rigorous system-theoretic analysis of LLMs is distinguishing between the discrete representation of symbolic sequences, the spatial feedforward propagation of continuous tensors through network depth, and the discrete-time recursions governing text generation and~tool~execution (see Fig.~\ref{fig:architecture}).

\subsection{Symbolic Sequences and Context Horizons}
Let $\Voc = \{v_1, v_2, \dots, v_V\}$ be a finite alphabet containing $V = |\Voc|$ distinct symbolic tokens (e.g., $V = 50{,}257$ for the Byte Pair Encoding (BPE) tokenizer of GPT-2-small). The set of all token sequences of fixed length $n \in \mathbb{N}$ is denoted by the Cartesian product $\Voc^n$, whose elements are written as ordered tuples $s = (w_1, w_2, \dots, w_n)$ with each $w_i \in \Voc$.

The space of all finite-length token strings is the free~monoid $\Vstar = \bigcup_{n=0}^\infty \Voc^n$ equipped with the associative binary operation of concatenation $\circ : \Vstar \times \Vstar \to \Vstar$ and the empty sequence $\varnothing \in \Voc^0$ acting as the unique identity element ($s \circ \varnothing = \varnothing \circ s = s$).

To perform linear algebraic operations, each discrete symbolic token is embedded into a continuous feature space of dimension $d \in \mathbb{N}$ (e.g., $d = 768$ for the reference GPT-2-small architecture). Consequently, an ordered sequence of $n$ tokens is represented continuously by a sequence state matrix residing in the finite-dimensional Euclidean space $\R^{n \times d}$.

Physical computing architectures impose a finite upper bound on the number of tokens a model can simultaneously attend to in memory. Let $T \in \mathbb{N}$ denote this maximum context horizon (e.g., $T = 1{,}024$ for GPT-2-small). The space of admissible input contexts is the restricted alphabet union $\Vstar_T = \bigcup_{n=1}^T \Voc^n$.

When sequences grow beyond this physical capacity during iterative generation, the memory is governed by the context truncation operator $\Pi_T : \Vstar \to \Vstar_T$, which projects any arbitrary string onto its trailing suffix of length at~most~$T$:
% \vspace{-16pt}
\begin{equation}
\Pi_T(w_1, w_2, \dots, w_n) = (w_{\max(1, n-T+1)}, \dots, w_n).
\label{eq:trunc}
\end{equation}
The map $\Pi_T$ is idempotent ($\Pi_T \circ \Pi_T = \Pi_T$) and acts as a deterministic projection operator on the string monoid.

\subsection{The Three Coupled Dynamical Phases}
% Unlike standard dynamical systems evolving over a single clock, a generative agentic LLM comprises three hierarchically coupled dynamical systems operating over three mutually distinct index sets:
The continuous representation of each token is a row vector in a latent feature space of embedding dimension $d \in \mathbb{N}$ (e.g., $d = 768$ for the reference GPT-2-small architecture). At any fixed generation instant $t \in \mathbb{N}$, the current context sequence $s_t \in \Vstar_T$ of length $n$ is mapped to an initial continuous tensor $H^{(0)} \in \R^{n \times d}$. This tensor propagates through a cascade of $L$ transformer blocks $\mathcal{B}_\ell$ (e.g., $L = 12$ for GPT-2-small, see also Fig.~\ref{fig:architecture}), producing intermediate tensors $H^{(\ell)} \in \R^{n \times d}$ which propagates via state transition over the finite-horizon $\ell \in \{1, \dots, L\}$.

From the terminal row of $H^{(L)}$, a categorical probability kernel $p_\theta(\cdot | s_t)$ is evaluated, according to which~a~continuation token $w_{t+1}$ is sampled, and the context is updated~via:
% \vspace{-12pt}
    \begin{equation}
    s_{t+1} = \Pi_T(s_t \circ w_{t+1}).
    \label{eq:concat_recursion_intro}
    \end{equation}
% \\[-28pt]

This recursion defines an uncontrolled, discrete-time stochastic difference equation evolving over token clock $t$.

Furthermore, in agentic systems and when executing external tools, the above generation in mode $q_G$ is paused upon emitting a tool-call or action
token defined by an autonomous switching manifold (Fig.~\ref{fig:hybrid_automaton}). The system then transitions to external mode $q_{E}$ and executes external physical or computational tasks. The environment observation $y$ is injected back into the context via a state jump map~$s^+ = \Pi_T \big( \Phi_{\scalebox{0.6}{$q_{E}q_G$}}\!\big(\!s_{t_k}, \operatorname{enc}(y)\!\big)\!\big)$, and is marked by count index~$k \in \mathbb{N}$.

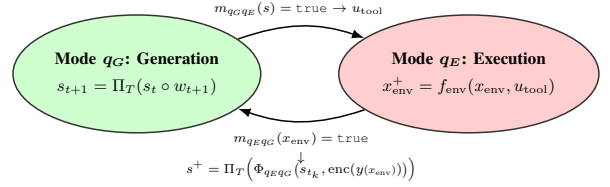
\begin{figure}[t]
\centering
\resizebox{0.92\linewidth}{!}{
\begin{tikzpicture}[
    font=\small, 
    >=Latex, 
    every edge/.style={draw, thick},
    mode_node/.style={
      ellipse, 
      draw=black!80, 
      thick,
      minimum width=3.6cm, 
      minimum height=2.2cm, 
      align=center, 
      inner sep=3pt
    }
  ]

  % Mode q_G: Generation
  \node[mode_node, fill=green!19] (qG) at (0,0) {
    \textbf{Mode $\bm{q_G}$: Generation}\\[4pt]
    $s_{t+1} = \Pi_T(s_t \circ w_{t+1})$
  };

  % Mode q_{E}: Execution
  \node[mode_node, fill=red!19] (qE) at (6.2,0) {
    \textbf{Mode $\bm{q_{E}}$: Execution}\\[4pt]
    $x_{\mathrm{env}}^{+} = f_{\mathrm{env}}(x_{\mathrm{env}}, u_{\mathrm{tool}})$
  };

  % Forward transition: Autonomous Switching Manifold Trigger (Curved Above)
  \path[->] (qG) edge [bend left=19] node[above=2pt, font=\scriptsize, align=center] {
    $m_{\scalebox{0.8}{$q_G q_{E}$}}(s) = \texttt{true}$ $\rightarrow$ $u_{\mathrm{tool}}$
  } (qE);

  % Return transition: Reset / Jump Map (Curved Below)
  \path[->] (qE) edge [bend left=19] node[below=2pt, font=\scriptsize, align=center] {$m_{\scalebox{0.8}{$q_{E}q_G$}}(x_{\mathrm{env}}) = \texttt{true}$
    \\[2pt]
    $\downarrow$\\[-3pt]
    $s^+ = \Pi_T\Big( \Phi_{\scalebox{0.8}{$q_{E}q_G$}}\big(s_{t_k}, \operatorname{enc}(y\scalebox{0.8}{$(x_{\mathrm{env}})$})\big)\Big)$
  } (qG);

\end{tikzpicture}
}
\vspace{-2mm}
\caption{Stochastic hybrid automata governing agentic LLM tool interaction, consisting of discrete operational modes $Q=\{q_G, q_{E}\}$, autonomous switchings activated upon arrivals at switching manifolds, and context augmentation for LLM via the state jump map $s^+ = \Pi_T\big( \Phi_{q_{E}q_G}\big(s_{t_k}, \operatorname{enc}(y\scalebox{0.8}{$(x_{\mathrm{env}})$})\big)\big)$.}
\label{fig:hybrid_automaton}
\vspace{-4mm}
\end{figure}

%=======================================================================
\section{The Deterministic Layerwise Cascade of Transformer Blocks}
\label{sec:cascade}
%=======================================================================
To formalize the state transition map $\mathcal{B}_\ell$ introduced in Section~\ref{sec:spaces}, this section details the internal dynamics of the transformer architecture. We first construct the algebraic embedding that maps the discrete sequence $s_t$ to the continuous state $H^{(0)}$. We then characterize the spherical projection geometry of layer normalization, formulate the spatial and channel mixing operations performed by multi-head attention and position-wise MLPs, and establish the exactness of Key--Value caching via causal prefix invariance.

\subsection{Input Embedding Operator and Initial State}
Let $e_i \in \R^V$ denote the canonical standard basis vector whose $i$-th entry is 1 and all other entries are 0. The discrete symbolic string $s_t = (w_1, \dots, w_n) \in \Voc^n$ of context length $n \le T$ is represented algebraically by a canonical selection matrix $U(s_t) \in \{0, 1\}^{n \times V}$:
\begin{equation}
U(s_t) \equiv U\scalebox{0.8}{\!$(w_1,\! ... , w_n)$\!} = \begin{bmatrix} e_{\operatorname{idx}(w_1)}^\top \\ \vdots \\ e_{\operatorname{idx}(w_n)}^\top \end{bmatrix},
\label{eq:selection_matrix}
\end{equation}
where $\operatorname{idx} : \Voc \to \{1, \dots, V\}$ returns the corresponding index of token $w_i$ in the vocabulary set $\Voc$.

Similarly, the spatial positions of tokens within the context window are represented by the leading $n$ canonical basis vectors of $\R^T$, denoted by $p_1, \dots, p_n \in \R^T$. Stacking these yields the positional selection matrix:
\begin{equation}
P_n = \begin{bmatrix} p_1^\top \\ p_2^\top \\ \vdots \\ p_n^\top \end{bmatrix} \equiv \begin{bmatrix} I_n & 0_{n \times (T-n)} \end{bmatrix} \in \{0, 1\}^{n \times T}.
\label{eq:pos_selection_matrix}
\end{equation}

The static learned parameters of the embedding interface comprise the token embedding matrix $W_e \in \R^{V \times d}$ and the absolute positional embedding matrix $W_p \in \R^{T \times d}$ (e.g., $W_e \in \R^{50{,}257 \times 768}$ and $W_p \in \R^{1{,}024 \times 768}$ for GPT-2-small). The initial continuous tensor state $H^{(0)} \in \R^{n \times d}$ entering the depth cascade is constructed as the sum of token and position representations:
\begin{equation}
H^{(0)} = U(s_t) W_e + P_n W_p.
\label{eq:embed_matrix}
\end{equation}
Row $i \in \{1, \dots, n\}$ of $H^{(0)}$, denoted by $h_i^{(0)\top} \in \R^{1 \times d}$, represents the continuous initial state of token position $i$ prior to inter-token interaction.

\subsection{Layer Normalization as an Oblique Spherical Projection}
Modern decoder-only architectures adopt the \emph{pre-layer-normalization} (Pre-LN) convention, applying normalization immediately prior to the attention and MLP sublayers. For any continuous token state $h \in \R^d$, the layer normalization operator $\LN_{\gamma, \beta} : \R^d \to \R^d$ is defined by:
\begin{equation}
\LN_{\gamma, \beta}(h) = \frac{h - \mu(h)\mathbf{1}_d}{\sigma(h)} \odot \gamma + \beta,
\label{eq:ln_def}
\end{equation}
where $\gamma, \beta \in \R^d$ are learned affine gain and bias parameters, $\odot$ denotes the elementwise Hadamard product, and the sample mean and standard deviation operators are:
\begin{equation}
\mu(h) = \frac{1}{d} \mathbf{1}_d^\top h, \quad \sigma(h) = \sqrt{\frac{1}{d} \|(I_d - \tfrac{1}{d}\mathbf{1}_d\mathbf{1}_d^\top)h\|_2^2 + \epsilon},
\label{eq:mu_sigma}
\end{equation}
with numerical regularizer $\epsilon > 0$ (e.g., $\epsilon = 10^{-5}$ for GPT-2-small).

From a geometric control perspective, $\LN$ is not an invertible coordinate transformation, but a nonlinear projection onto an embedded $(d-2)$-sphere in $\R^d$:

\begin{proposition}[Spherical Projection Geometry of Layer Normalization]
\label{prop:ln_sphere}
Let $\Pi_{\mathbf{1}}^\perp = I_d - \frac{1}{d}\mathbf{1}_d \mathbf{1}_d^\top$ denote the orthogonal projection matrix onto the zero-mean hyperplane $\mathbf{1}_d^\perp = \{x \in \R^d : \mathbf{1}_d^\top x = 0\}$. In the asymptotic limit $\epsilon \to 0$, the unparameterized normalization map $\LN_{0}(h) = \frac{\Pi_{\mathbf{1}}^\perp h}{\frac{1}{\sqrt{d}}\|\Pi_{\mathbf{1}}^\perp h\|_2}$ projects all states $h \notin \operatorname{span}\{\mathbf{1}_d\}$ onto the centered Euclidean sphere of radius $\sqrt{d}$ embedded in $\mathbf{1}_d^\perp$:
\begin{equation}
\mathcal{S}^{d-2}_{\sqrt{d}} = \left\{ z \in \R^d : \mathbf{1}_d^\top z = 0, \ \|z\|_2 = \sqrt{d} \right\}.
\end{equation}
Consequently, layer normalization is scale- and shift-invariant, satisfying $\LN_{0}(\alpha h + \beta \mathbf{1}_d) = \LN_{0}(h)$ for all $\alpha \in \R_{> 0}$ and $\beta \in \R$.
\hfill $\square$ \end{proposition}
\begin{proof}
By definition, $\mathbf{1}_d^\top (\Pi_{\mathbf{1}}^\perp h) = (\mathbf{1}_d^\top - \frac{d}{d}\mathbf{1}_d^\top)h = 0$, which ensures $\mu(\LN_0(h)) = 0$. Computing the Euclidean norm:
\begin{equation}
\|\LN_0(h)\|_2^2 = \frac{\|\Pi_{\mathbf{1}}^\perp h\|_2^2}{\frac{1}{d}\|\Pi_{\mathbf{1}}^\perp h\|_2^2} = d \implies \|\LN_0(h)\|_2 = \sqrt{d}.
\end{equation}
Scale invariance follows since $\Pi_{\mathbf{1}}^\perp (\alpha h + \beta \mathbf{1}_d) = \alpha \Pi_{\mathbf{1}}^\perp h$, and the scalar $\alpha > 0$ cancels identically in the numerator and denominator.
\end{proof}

For a sequence state matrix $H \in \R^{n \times d}$, row-wise normalization is denoted by $$\LN(H) = [\LN(h_1), \dots, \LN(h_n)]^\top \in \R^{n \times d}.$$

\subsection{Multi-Head Causal Self-Attention (Spatial Mixing)}
The primary mechanism enabling spatial communication across token positions is multi-head causal self-attention. Inside block $\mathcal{B}_\ell$, the input state $H^{(\ell-1)} \in \R^{n \times d}$ is first normalized to yield $Z^{(\ell)} = \LN_{\ell, 1}(H^{(\ell-1)})$. The normalized state is projected into $N_h$ parallel attention heads of latent dimension $d_k = d / N_h$ (for GPT-2-small, $N_h = 12$ and $d_k = 768 / 12 = 64$, see also Fig.~\ref{fig:architecture}).

For head $m \in \{1, \dots, N_h\}$, Query, Key, and Value matrices $Q^{(\ell, m)}, K^{(\ell, m)}, V^{(\ell, m)} \in \R^{n \times d_k}$ are generated via learned projection parameters $W_Q^{(\ell, m)}, W_K^{(\ell, m)}, W_V^{(\ell, m)} \in \R^{d \times d_k}$ and bias vectors $b_Q^{(\ell, m)}, b_K^{(\ell, m)}, b_V^{(\ell, m)} \in \R^{d_k}$:
\begin{equation}
\begin{aligned}
Q^{(\ell, m)} &= Z^{(\ell)} W_Q^{(\ell, m)} + \mathbf{1}_n b_Q^{(\ell, m)\top}, \\
K^{(\ell, m)} &= Z^{(\ell)} W_K^{(\ell, m)} + \mathbf{1}_n b_K^{(\ell, m)\top}, \\
V^{(\ell, m)} &= Z^{(\ell)} W_V^{(\ell, m)} + \mathbf{1}_n b_V^{(\ell, m)\top}.
\end{aligned}
\label{eq:qkv_proj}
\end{equation}

Inter-token affinities are computed through bilinear inner products scaled by $1/\sqrt{d_k}$. To enforce causality in decoder-only sequence generation (preventing position $i$ from attending to future positions $j > i$), an additive causal mask $M_n \in (\R \cup \{-\infty\})^{n \times n}$ is enforced:
\begin{equation}
[M_n]_{i,j} = \begin{cases} 0, & i \ge j, \\ -\infty, & i < j. \end{cases}
\label{eq:causal_mask}
\end{equation}
The masked pre-softmax affinity matrix $S^{(\ell, m)} \in \R^{n \times n}$ is:
\begin{equation}
S^{(\ell, m)} = \frac{Q^{(\ell, m)} (K^{(\ell, m)})^\top}{\sqrt{d_k}} + M_n.
\label{eq:scores}
\end{equation}

Applying the row-wise softmax operator produces the stochastic attention routing matrix $A^{(\ell, m)} \in \R^{n \times n}$:
\begin{equation}
[A^{(\ell, m)}]_{i,j} = \frac{\exp([S^{(\ell, m)}]_{i,j})}{\sum_{p=1}^i \exp([S^{(\ell, m)}]_{i,p})}, \quad j \le i,
\label{eq:attention_matrix}
\end{equation}
and $[A^{(\ell, m)}]_{i,j} = 0$ for all $j > i$. Thus, $A^{(\ell, m)}$ is lower triangular and strictly row-stochastic ($A^{(\ell, m)} \mathbf{1}_n = \mathbf{1}_n$, $[A^{(\ell, m)}]_{i,j} \ge 0$).

Each head mixes values along the spatial sequence dimension: $O^{(\ell, m)} = A^{(\ell, m)} V^{(\ell, m)} \in \R^{n \times d_k}$. The heads are concatenated along the feature dimension ($d$) and mapped through an output projection matrix $W_O^{(\ell)} \in \R^{d \times d}$ and bias $b_O^{(\ell)} \in \R^d$:
\begin{equation}
\MHA_\ell(Z^{(\ell)}) = \! \Big( \! \Concat_{m=1}^{N_h} \! \big[ A^{(\ell, m)} V^{(\ell, m)} \big] \! \Big) W_O^{(\ell)} + \mathbf{1}_n b_O^{(\ell)\top} \!.
\label{eq:mha_output}
\end{equation}

The intermediate state $\widetilde{H}^{(\ell)}$ is formed via the first residual stream connection:
\begin{equation}
\widetilde{H}^{(\ell)} = H^{(\ell-1)} + \MHA_\ell\big(\LN_{\ell, 1}(H^{(\ell-1)})\big).
\label{eq:residual_attn}
\end{equation}

\subsection{Position-Wise Multi-Layer Perceptron (Channel Mixing)}
While self-attention mixes state information along the spatial sequence axis ($n$), the position-wise Multi-Layer Perceptron (MLP) acts independently on each token row, performing nonlinear channel mixing across the latent feature dimension ($d$).

The intermediate state $\widetilde{H}^{(\ell)}$ is normalized via a second layer norm to yield $R^{(\ell)} = \LN_{\ell, 2}(\widetilde{H}^{(\ell)})$. The MLP expands the channel dimension by an expansion factor of 4 to $d_{\mathrm{mlp}} = 4d$ (for GPT-2-small, $d_{\mathrm{mlp}} = 3{,}072$), applies an elementwise activation function, and projects back to dimension $d$:
\begin{equation}
\MLP_\ell(R^{(\ell)}) = \sigma_{\operatorname{act}}\big( R^{(\ell)} W_{\mathrm{fc}}^{(\ell)} + \mathbf{1}_n b_{\mathrm{fc}}^{(\ell)\top} \big) W_{\mathrm{proj}}^{(\ell)} + \mathbf{1}_n b_{\mathrm{proj}}^{(\ell)\top},
\label{eq:mlp_matrix}
\end{equation}
where $W_{\mathrm{fc}}^{(\ell)} \in \R^{d \times 4d}$, $b_{\mathrm{fc}}^{(\ell)} \in \R^{4d}$, $W_{\mathrm{proj}}^{(\ell)} \in \R^{4d \times d}$, and $b_{\mathrm{proj}}^{(\ell)} \in \R^d$ are all pre-trained. In canonical transformer architectures, $\sigma_{\operatorname{act}}$ is the Gaussian Error Linear Unit ($\GELU$):
\begin{equation}
\GELU(x) =  x \! \int\limits_{-\infty}^x \! \frac{1}{\sqrt{2\pi}} e^{\scalebox{0.8}{$\frac{-s^2}{2}$}} \, \mathrm{d}s \equiv x \left( \frac{1}{2} \left[ 1 + \operatorname{erf}\left( \frac{x}{\sqrt{2}} \right) \right] \right).
\label{eq:gelu_def}
\end{equation}

The terminal block output is obtained through the second residual stream:
\begin{equation}
H^{(\ell)} = \widetilde{H}^{(\ell)} + \MLP_\ell\big(\LN_{\ell, 2}(\widetilde{H}^{(\ell)})\big).
\label{eq:residual_mlp}
\end{equation}

% \vspace{-4pt}

Equations \eqref{eq:residual_attn} and \eqref{eq:residual_mlp} define the block state transition map $\mathcal{B}_\ell : \R^{n \times d} \to \R^{n \times d}$, yielding the complete depth cascade:
% \begin{equation}
% H^{(\ell)} = \mathcal{B}_\ell(H^{(\ell-1)}), \quad \ell = 1, \dots, L.
% \label{eq:depth_cascade}
% \end{equation}
\begin{multline}
\hspace{-9pt} H^{(\ell)} = \mathcal{B}_\ell\big(H^{(\ell-1)}\big) \triangleq H^{(\ell-1)} + \MHA_\ell\big(\LN_{\ell,1}(H^{(\ell-1)})\big) 
\\
+ \MLP_\ell\!\left(\LN_{\ell,2}\Big(H^{(\ell-1)} + \MHA_\ell\big(\LN_{\ell,1}(H^{(\ell-1)})\big)\Big)\right)\!.\hspace{-9pt}
\label{eq:depth_cascade}
\end{multline}

% \vspace{-9pt}

\subsection{Unembedding Readout and Logit Generation}
After propagating through all $L$ transformer blocks, the sequence state $H^{(L)} \in \R^{n \times d}$ is processed by a final layer normalization operator $\LN_f$ and projected onto the vocabulary space via the unembedding matrix $W_u \in \R^{d \times V}$ (with $W_u = W_e^\top \in \R^{768 \times 50{,}257}$ under parameter weight-tying):
\begin{equation}
Z_{\mathrm{logits}}(s_t) = \LN_f(H^{(L)}) W_u \in \R^{n \times V}.
\label{eq:logits_readout}
\end{equation}
Because autoregressive sequence continuation depends strictly on predicting the successor to the current context string, only the terminal row $n$ of the logit matrix is~read~out:
\begin{equation}
z(s_t) = e_n^\top Z_{\mathrm{logits}}(s_t) \in \R^{1 \times V}.
\label{eq:terminal_logits}
\end{equation}
The logit vector $z(s_t)$ parameterizes the next-token probability distribution over alphabet $\Voc$ via the parameterized temperature-scaled softmax operator:
\begin{equation}
p_\theta(w_{n+1} = v_j \mid s_t) = \frac{\exp(z_j(s_t) / \tau)}{\sum_{k=1}^V \exp(z_k(s_t) / \tau)},
\label{eq:temp_distribution}
\end{equation}
where $\tau > 0$ is the generation temperature ($\tau = 1.0$ under standard nominal sampling).

\subsection{Key--Value Caching as an Exact Internal State Realization}
In computational practice, autoregressive token generation across context length $n$ is accelerated via Key--Value (KV) caching to avoid re-evaluating preceding sequence states across the depth cascade. In the machine learning literature, this procedure is routinely treated as an algorithmic runtime optimization or an implementation detail of the \mbox{attention operator}.

From a systems theory perspective, however, the KV cache fulfills the classical role of an internal state: a perpetually renewing initial condition that absorbs the entire history of past inputs into a minimal sufficient statistic. Rather than maintaining an expanding memory operator over the full historical sequence, causal prefix invariance ensures that the cached keys and values decouple past trajectory history from future evolution through an exact, instantaneous readout.

\begin{proposition}[Internal State Realization via Causal Prefix Invariance]
\label{prop:kv_state_realization}
Let $s = (w_1, \dots, w_n) \in \Voc^n$ be an admissible context sequence and let $s \circ w$ denote its single-token continuation with $w \in \Voc$. 
\\[5pt]
\noindent \textbf{1) Causal Prefix Invariance:} For every block $\ell \in \{1, \dots, L\}$ and token position $i \in \{1, \dots, n\}$, the hidden row states are invariant under future token arrivals:
\begin{equation}
e_i^\top H^{(\ell)}(s \circ w) = e_i^\top H^{(\ell)}(s).
\label{eq:prefix_invariance}
\end{equation}

\noindent \textbf{2) Exact State Realization:} The accumulated key and value tensors across all layers and heads,
\begin{equation}
X_n(s) \triangleq \left\{ \big(K^{(\ell,m)}(s),\,V^{(\ell,m)}(s)\big)\right\}_{\begin{subarray}{c} m\in\{1,\dots,N_{h}\} \\ \ell\in\{1,\dots,L\} \end{subarray}},
\label{eq:kv_state_tuple}
\end{equation}
constitute an exact internal state of the system:
\begin{itemize}
\item \textbf{Recursive State Transition:} The successor state $X_{n+1}$ is governed by a deterministic state flow map driven solely by the current state and incoming token~$w$:
\begin{equation}
X_{n+1} = \varphi(X_n, w).
\label{eq:exact_equivalence_state}
\end{equation}
\item \textbf{Instantaneous Readout:} The terminal logit vector \mbox{$z(s \circ w)$} defined through \eqref{eq:terminal_logits} admits an instantaneous readout map $\eta$ evaluated strictly from $X_n$ and $w$:
\begin{equation}
z(s \circ w) = \eta(X_n, w) \equiv z_{\mathrm{full}}(s \circ w),
\label{eq:exact_equivalence_output}
\end{equation}
yielding bit-level mathematical identity with full sequence re-evaluation. \hfill $\square$ 
\end{itemize}

\end{proposition}

\begin{proof}
For Part~1, appending $w \in \Voc$ to $s \in \Voc^n$ augments the selection matrices \eqref{eq:selection_matrix}--\eqref{eq:pos_selection_matrix} with an additional bottom row: $U(s \circ w) = [U(s)^\top, \, e_{\operatorname{idx}(w)}]^\top$ and $P_{n+1} = [P_n^\top, \, p_{n+1}]^\top$. Substituting these into \eqref{eq:embed_matrix} establishes the base block structure of the initial continuous state:
\begin{equation}
H^{(0)}(s \circ w) = \begin{bmatrix} U(s) W_e + P_n W_p \\ e_{\operatorname{idx}(w)}^\top W_e + p_{n+1}^\top W_p \end{bmatrix} = \begin{bmatrix} H^{(0)}(s) \\ h_{n+1}^{(0)\top} \end{bmatrix}.
\label{eq:h0_partition}
\end{equation}
Now consider any state matrix partitioned as $H = [H_n^\top, \, h_{n+1}]^\top \in \R^{(n+1) \times d}$. By \eqref{eq:mu_sigma}, $\mu(e_i^\top H) = \mu(e_i^\top H_n)$ and $\sigma^2(e_i^\top H) = \sigma^2(e_i^\top H_n)$ for all $i \le n$, which ensures that row-wise layer normalization \eqref{eq:ln_def} preserves the partition: $Z^{(\ell)} = [Z_n^{(\ell)\top}, \, z_{n+1}^{(\ell)}]^\top$. Evaluating the affine projections \eqref{eq:qkv_proj} directly yields $K^{(\ell, m)} = [K_n^{(\ell, m)\top}, \, k_{n+1}^{(\ell, m)}]^\top$, with similar partitions holding for $Q^{(\ell, m)}$ and $V^{(\ell, m)}$.

In the causal affinity matrix \eqref{eq:scores}, the causal mask \eqref{eq:causal_mask} enforces $[M_{n+1}]_{i, n+1} = -\infty$ for all $i \le n$, which yields:
\begin{multline}
S^{(\ell, m)} = \begin{bmatrix} \frac{1}{\sqrt{d_k}} Q_n^{(\ell, m)} K_n^{(\ell, m)\top} + M_n & -\boldsymbol{\infty}_{n \times 1} \\ \frac{1}{\sqrt{d_k}} q_{n+1}^{(\ell, m)\top} K_n^{(\ell, m)\top} & \frac{1}{\sqrt{d_k}} q_{n+1}^{(\ell, m)\top} k_{n+1}^{(\ell, m)} \end{bmatrix} 
\\
= \begin{bmatrix} S_n^{(\ell, m)} & -\boldsymbol{\infty}_{n \times 1} \\ s_{\mathrm{cross}}^\top & s_{n+1, n+1} \end{bmatrix}.
\end{multline}
Applying the row-wise softmax \eqref{eq:attention_matrix} annihilates column $n+1$ over historical rows via $\exp(-\infty) = 0$, giving $A^{(\ell, m)} = \begin{bmatrix} A_n^{(\ell, m)} & \mathbf{0}_{n \times 1} \\ \alpha_{n+1}^{(\ell, m)\top} & a_{n+1, n+1}^{(\ell, m)} \end{bmatrix}$. The value mixing in \eqref{eq:mha_output} then evaluates to:
\begin{multline}
\hspace{-9pt} O^{(\ell, m)} = A^{(\ell, m)} V^{(\ell, m)} 
\\
= \begin{bmatrix} A_n^{(\ell, m)} V_n^{(\ell, m)} + \mathbf{0}_{n \times 1} v_{n+1}^{(\ell, m)\top} \\ \alpha_{n+1}^{(\ell, m)\top} V_n^{(\ell, m)} + a_{n+1, n+1}^{(\ell, m)} v_{n+1}^{(\ell, m)\top} \end{bmatrix} = \begin{bmatrix} O_n^{(\ell, m)} \\ o_{n+1}^{(\ell, m)\top} \end{bmatrix} \!. \hspace{-7pt}
\end{multline}
Because the head projection \eqref{eq:mha_output}, residual stream \eqref{eq:residual_attn}, second layer normalization \eqref{eq:ln_def}, MLP channel mixing \eqref{eq:mlp_matrix}, and residual addition \eqref{eq:residual_mlp} act strictly row-wise, the block transition map satisfies $\mathcal{B}([H_n^\top, \, h_{n+1}]^\top) = [\mathcal{B}(H_n)^\top, \, \tilde{h}_{n+1}]^\top$ as~well. Applying this block-lower triangular operator across the cascade \eqref{eq:depth_cascade} starting from \eqref{eq:h0_partition} yields that for all $\ell \in \{1, \dots, L\}$:
\begin{equation}
H^{(\ell)}(s \circ w) = \begin{bmatrix} H^{(\ell)}(s) \\ h_{n+1}^{(\ell)\top} \end{bmatrix},
\end{equation}
which establishes that causal prefix invariance $e_i^\top H^{(\ell)}(s \circ w)$ $= e_i^\top H^{(\ell)}(s)$ holds for all $i \le n$.

For Part~2, note that the terminal logit vector $z(s \circ w)$ is extracted from the $(n+1)$-st row of the final block via \eqref{eq:terminal_logits}:
\begin{equation}
z(s \circ w) = e_{n+1}^\top Z_{\mathrm{logits}}(s \circ w) = \LN_f(h_{n+1}^{(L)\top}) W_u.
\label{eq:readout_target}
\end{equation}
We now show that the row state $h_{n+1}^{(\ell)}$ at each depth $\ell \in \{1, \dots, L\}$ is evaluated recursively using only $h_{n+1}^{(\ell-1)}$ and the historical state cache $X_n(s)$ in \eqref{eq:kv_state_tuple}. At block $\ell$, the normalized query vector is:
\begin{equation}
z_{n+1}^{(\ell)\top} \!= \LN_{\gamma, \beta}^{(1)}(h_{n+1}^{(\ell-1)\top}), \quad q_{n+1}^{(\ell, m)\top} \!= z_{n+1}^{(\ell)\top} W_Q^{(\ell, m)} + b_Q^{(\ell, m)\top} \!.
\end{equation}
The row vector of attention weights assigned by the \mbox{$(n+1)$-st} token across historical positions $j \le n$ is given by the \mbox{$(n+1)$-st} row of \eqref{eq:attention_matrix}:
\begin{equation}
\alpha_{n+1}^{(\ell, m)\top} \!\!= \operatorname{softmax}\!\Big(\! \frac{1}{\sqrt{d_k}} \Big[ q_{n+1}^{(\ell, m)\top}\! K_n^{(\ell, m)\top}, \; q_{n+1}^{(\ell, m)\top}\! k_{n+1}^{(\ell, m)} \Big] \Big)_{\!1:n},
\label{eq:alpha_eval}
\end{equation}
where $k_{n+1}^{(\ell, m)\top} = z_{n+1}^{(\ell)\top} W_K^{(\ell, m)} + b_K^{(\ell, m)\top}$ and $v_{n+1}^{(\ell, m)\top} = z_{n+1}^{(\ell)\top} W_V^{(\ell, m)} + b_V^{(\ell, m)\top}$. By causal prefix invariance established in Part~1, the historical matrices in \eqref{eq:alpha_eval} satisfy:
\begin{equation}
K_n^{(\ell, m)} \equiv K^{(\ell, m)}(s), \qquad V_n^{(\ell, m)} \equiv V^{(\ell, m)}(s),
\end{equation}
which are identically the components of the cached state tuple $X_n(s)$ in \eqref{eq:kv_state_tuple}. The mixed value vector for the \mbox{$(n+1)$-st} token in \eqref{eq:mha_output} then evaluates directly to:
\begin{equation}
o_{n+1}^{(\ell, m)\top} = \alpha_{n+1}^{(\ell, m)\top} K_n^{(\ell, m)} V_n^{(\ell, m)} + a_{n+1, n+1}^{(\ell, m)} v_{n+1}^{(\ell, m)\top}.
\end{equation}
Concatenating heads and passing through the output projection in \eqref{eq:mha_output} yields the intermediate residual state:
\begin{equation}
\bar{h}_{n+1}^{(\ell)\top} = h_{n+1}^{(\ell-1)\top} + \Big[ \operatorname{Concat}_{m=1}^{N_h} \big( o_{n+1}^{(\ell, m)\top} \big) \Big] W_O + b_O^\top,
\end{equation}
followed by the channel-mixing MLP update \eqref{eq:mlp_matrix}--\eqref{eq:residual_mlp}:
\begin{equation}
h_{n+1}^{(\ell)\top} = \bar{h}_{n+1}^{(\ell)\top} + \operatorname{GELU}\big( \LN_{\gamma, \beta}^{(2)}(\bar{h}_{n+1}^{(\ell)\top}) W_{\mathrm{fc}} + b_{\mathrm{fc}}^\top \big) W_{\mathrm{proj}} + b_{\mathrm{proj}}^\top.
\label{eq:h_terminal_step}
\end{equation}
Because the state update $X_{n+1} = \varphi(X_n, w)$ simply appends the newly computed vectors $[k_{n+1}^{(\ell, m)\top}, \, v_{n+1}^{(\ell, m)\top}]$ to $X_n$ at each layer without altering existing entries, the forward propagation \eqref{eq:alpha_eval}--\eqref{eq:h_terminal_step} produces the terminal vector $h_{n+1}^{(L)\top}$ identically to that obtained by evaluating the full sequence from scratch. Applying the final unembedding map \eqref{eq:readout_target}~gives:
\begin{equation}
\eta(X_n, w) \triangleq \LN_f(h_{n+1}^{(L)\top}) W_u \equiv z_{\mathrm{full}}(s \circ w),
\end{equation}
which establishes the exactness of the state-space realization~\eqref{eq:exact_equivalence_state}--\eqref{eq:exact_equivalence_output}.
\end{proof}

%=======================================================================
\section{The Uncontrolled Generation Recursion}
\label{sec:generation}
%=======================================================================
Across the generation clock $t \in \mathbb{N}$, sequence rollout operates as the unforced discrete-time stochastic recursion~\eqref{eq:concat_recursion_intro}. In the absence of exogenous reference or inputs, the state trajectory evolves on $\mathcal{V}_T^*$ driven strictly by internal stochastic sampling from \eqref{eq:temp_distribution} under the context projection $\Pi_T$.

%-----------------------------------------------------------------------
\subsection{Markovian Structure and Deterministic Periodicity}
\label{sec:markov_periodicity}

\begin{proposition}[Finite-State Markov Representation]
\label{prop:markov_chain}
Under context truncation $\Pi_T$, the unforced recursion \eqref{eq:concat_recursion_intro} defines a time-homogeneous Markov chain on the finite state space $\mathcal{V}_T^*$ with transition matrix $\bm{\mathcal{P}} \in \mathbb{R}^{|\mathcal{V}_T^*| \times |\mathcal{V}_T^*|}$:
\begin{equation}
\bm{\mathcal{P}}_{s, \tilde{s}} = 
\begin{cases}
p_\theta(w \mid s), & \text{if } \exists w \in \mathcal{V} \text{ s.t. } \tilde{s} = \Pi_T(s \circ w), \\[2pt]
0, & \text{otherwise},
\end{cases}
\label{eq:markov_transition_kernel}
\end{equation}
and admits at least one stationary distribution $\pi^* \in \Delta^{|\mathcal{V}_T^*|-1}$ satisfying $\pi^* \bm{\mathcal{P}} = \pi^*$.
\hfill $\square$ \end{proposition}
\begin{proof}
Since both the sampling distribution $p_\theta(\cdot | s_t)$ and the projection $\Pi_T$ depend solely on the current state $s_t$, the sequence $\{s_t\}_{t \ge 0}$ is stationary and Markovian. 

Moreover, because $|\mathcal{V}_T^*| = \sum_{n=1}^T V^n < \infty$, the existence of $\pi^*$ follows directly from the Perron--Frobenius theorem on the compact simplex $\Delta^{|\mathcal{V}_T^*|-1}$.
\end{proof}

When the sampling temperature vanishes ($\tau \to 0^+$), stochasticity collapses. Since the pre-trained weights are fixed across generation steps, the system reduces to an unforced, time-invariant, deterministic finite-state machine.

\begin{proposition}[Ultimate Periodicity of Greedy Orbits]
\label{prop:greedy_periodicity}
Let greedy decoding be denoted by $s_{t+1} = \mathcal{F}_{\mathrm{greedy}}(s_t) \triangleq \Pi_T(s_t \circ \argmax_{v \in \mathcal{V}} z_v(s_t))$ with deterministic tie-breaking. For any prompt $s_0 \in \mathcal{V}_T^*$, the state orbit $\{s_t\}_{t \ge 0}$ is ultimately periodic, i.e., there exist $t_0 < |\mathcal{V}_T^*|$ and a period $T_{\mathrm{per}} \le |\mathcal{V}_T^*|$ such that:
\begin{equation}
s_{t + T_{\mathrm{per}}} = s_t, \qquad \forall t \ge t_0.
\label{eq:periodic_limit_cycle}
\end{equation}
\hfill $\square$ \end{proposition}
\begin{proof}
Because $\mathcal{V}_T^*$ is finite, any orbit of length $|\mathcal{V}_T^*| + 1$ must revisit a state by the pigeonhole principle: $s_{t_1} = s_{t_2}$ for some $0 \le t_1 < t_2 \le |\mathcal{V}_T^*|$. Since $\mathcal{F}_{\mathrm{greedy}}$ is time-invariant and deterministic, the trajectories initiated from $s_{t_1}$ and $s_{t_2}$ are identical: $s_{t_1+k} = s_{t_2+k}$ for all $k \ge 0$. Setting $t_0 = t_1$ and $T_{\mathrm{per}} = t_2 - t_1$ completes the proof.
\end{proof}

Proposition~\ref{prop:greedy_periodicity} shows that repetitive looping in unforced greedy decoding is a direct consequence of deterministic difference equations evolving on finite state spaces.

%-----------------------------------------------------------------------
%-----------------------------------------------------------------------

%-----------------------------------------------------------------------
%-----------------------------------------------------------------------
\subsection{Task-Level Error Processes}
\label{sec:error_drift}

Proposition~\ref{prop:markov_chain} characterizes the stochastic
evolution of~$s_t$, but not task correctness.
Metrics on neither $s_t \in \mathcal{V}_T^*$ nor its embedded state
$H^0\in\mathbb{R}^{n\times d}$ in \eqref{eq:embed_matrix} generally
measure task error. Let $(\mathcal{X},d_{\mathcal{X}})$ be a metric
task space, let
\mbox{$\mathcal{S}_{\mathrm{eval}}(s_0)\subset\mathcal{V}_T^*$} contain
contexts with an evaluable claim, and let
$\Psi:\mathcal{S}_{\mathrm{eval}}(s_0)\to\mathcal{X}$ be a
task-dependent readout.

Let $\tau_1<\tau_2<\cdots$ be the successive generation instants for
which $s_{\tau_j}\in\mathcal{S}_{\mathrm{eval}}(s_0)$. Define
\begin{equation}
\mathsf{e}_j\triangleq
d_{\mathcal{X}}\!\left(
    \Psi(s_{\tau_j}),x^*_{s_0,j}
\right),
\qquad
\mathcal{G}_j\triangleq
\sigma(s_0,w_1,\ldots,w_{\tau_j}),
\label{eq:epistemic_error_def}
\end{equation}
where $x^*_{s_0,j}\in\mathcal{X}$ is the reference value for the
$j$th evaluated claim. If $\Psi$ identifies a malformed claim, set
$\mathsf{e}_j=\infty$. Incomplete prefixes have no associated error
value. Thus, $\mathsf{e}_j$ measures the current evaluated claim and
may decrease after self-correction.

The transition kernel \eqref{eq:markov_transition_kernel}, together
with $\Psi$, induces the error process $\{\mathsf{e}_j\}$. The
properties below concern a fixed model, decoding rule, prompt class,
and evaluator; they do not follow from the architecture alone.

For each $j\geq1$, let
\begin{equation}
\mathsf{E}_j(s_0)
\triangleq
\bigcap_{i=1}^{j}
\left\{
d_{\mathcal{X}}\!\left(
    \Psi(s_{\tau_i}),x^*_{s_0,i}
\right)=0
\right\}
=
\bigcap_{i=1}^{j}\{\mathsf{e}_i=0\}.
\label{eq:fault_free_event}
\end{equation}
Hence, $\mathsf{E}_j(s_0)$ is the event that the first $j$ evaluated
claims generated from $s_0$ are correct. Let $\mathsf{E}_0(s_0)$
denote the sure event.

\begin{assumption}[Conditional Fault Hazard]
\label{assump:fault_hazard}
Assume that there exists $\alpha\in(0,1)$ such that
\begin{equation}
\mathbb{P}\!\left(
\mathsf{e}_{j+1}>0\mid\mathcal{G}_j
\right)
\geq \alpha
\quad\text{a.s.\ on }\mathsf{E}_j(s_0).
\label{eq:fault_hazard}
\end{equation}
\end{assumption}

\begin{proposition}[Fault-Free Reliability]
\label{thm:geometric_decay}
Under Assumption~\ref{assump:fault_hazard}, for every
$M\in\mathbb{N}$,
\begin{equation}
\mathbb{P}\!\left(\mathsf{E}_M(s_0)\right)
\equiv
\mathbb{P}\!\left(
    \bigcap_{i=1}^{M}\{\mathsf{e}_i=0\}
\right)
\leq
(1-\alpha)^M.
\label{eq:geometric_bound}
\end{equation}
\hfill $\square$
\end{proposition}

\begin{proof}
Since
$\mathsf{E}_{j+1}(s_0)
=\mathsf{E}_j(s_0)\cap\{\mathsf{e}_{j+1}=0\}$, it follows from
Assumption~\ref{assump:fault_hazard} that
\begin{equation*}
\mathbb{P}\!\left(\mathsf{E}_{j+1}(s_0)\right)
\leq
(1-\alpha)\mathbb{P}\!\left(\mathsf{E}_j(s_0)\right).
\end{equation*}
Iteration from $\mathbb{P}(\mathsf{E}_0(s_0))=1$ proves
\eqref{eq:geometric_bound}.
\end{proof}

Proposition~\ref{thm:geometric_decay} bounds the probability of a
fault-free claim history, not terminal correctness. A later claim may
correct an earlier fault without making the preceding history
\mbox{fault-free}.

\begin{assumption}[Conditional Error Drift]
\label{assump:error_drift}
Over an evaluation horizon $J$, assume that
$\{\mathsf{e}_j\}_{j=1}^{J}$ is integrable and that there exists
$\delta_{\mathrm{drift}}\geq0$ such that
\begin{equation}
\mathbb{E}\!\left[
    \mathsf{e}_{j+1}-\mathsf{e}_j\mid\mathcal{G}_j
\right]
\geq
\delta_{\mathrm{drift}},
\qquad 1\leq j<J.
\label{eq:drift_condition}
\end{equation}
\end{assumption}

\begin{proposition}[Expected Error Growth]
\label{thm:submartingale_drift}
Under Assumption~\ref{assump:error_drift},
$\{\mathsf{e}_j\}_{j=1}^{J}$ is a submartingale. For $j+m\leq J$,
\begin{equation}
\mathbb{E}\!\left[
    \mathsf{e}_{j+m}\mid\mathcal{G}_j
\right]
\geq
\mathsf{e}_j+m\delta_{\mathrm{drift}}.
\label{eq:submartingale_divergence}
\end{equation}
\hfill $\square$
\end{proposition}

\begin{proof}
The process $\{\mathsf{e}_j\}$ is $\mathcal{G}_j$-adapted and
integrable by construction and Assumption~\ref{assump:error_drift}.
Since
\begin{equation*}
\mathbb{E}[\mathsf{e}_{j+1}\mid\mathcal{G}_j]
=
\mathsf{e}_j+
\mathbb{E}[
    \mathsf{e}_{j+1}-\mathsf{e}_j\mid\mathcal{G}_j
]
\geq
\mathsf{e}_j,
\end{equation*}
$\{\mathsf{e}_j,\mathcal{G}_j\}_{j=1}^{J}$ is a submartingale by
definition. Iterating \eqref{eq:drift_condition} by the tower property
gives \eqref{eq:submartingale_divergence}.
\end{proof}

The task-level quantities above separate architectural structure from
behavioral performance. The Markov kernel specifies the distribution
of generated contexts, whereas the evaluator $\Psi$ and the
assumptions on $\alpha$ and $\delta_{\mathrm{drift}}$ specify a
task-dependent reliability model. Their identification for particular
models, decoders, and task classes is left for future~work.

%=======================================================================
% %=======================================================================
% \section{Conclusion}
% \label{sec:conclusion}
% %=======================================================================

%=======================================================================
\section{Conclusion}
\label{sec:conclusion}
%=======================================================================

This paper establishes a control-theoretic model of decoder-only
language models. It separates the deterministic transformer cascade
across layer depth from the uncontrolled stochastic recursion across
token-generation time, and embeds both within a hybrid description of
agentic execution and context updates. The model makes the associated
state spaces, clocks, and transition mechanisms explicit.

Within this formulation, familiar system-theoretic questions admit
precise statements and direct analysis: Causal~prefix invariance makes
Key--Value caching an exact internal state realization. Context
truncation induces a time-homogeneous finite-state Markov
representation, while fixed parameters and deterministic tie-breaking
make greedy decoding ultimately periodic. At the task level, a
prompt-dependent evaluator induces an error process over successive
evaluable claims. Conditional fault-hazard and error-drift assumptions
then yield bounds on fault-free claim histories and expected error
growth, respectively.

The architectural results follow from the stated decoder dynamics.
The task-level results instead depend on an evaluator and on
model-, decoder-, and task-specific assumptions. Their identification,
the role of external information in the induced error process, and a
detailed stochastic hybrid-systems treatment of agentic language
models are left for future work.

\bibliographystyle{IEEEtran}
\bibliography{LLMReferences}

\end{document}